\documentclass[11pt]{article}

\usepackage[utf8]{inputenc}
\usepackage{etoolbox}
\usepackage{amssymb}
\usepackage{amsmath}
\usepackage{amsthm}
\usepackage{amscd}
\usepackage{amsfonts}
\usepackage{mathrsfs}
\usepackage{mathtools}
\usepackage{mathabx}
\usepackage{enumerate}
\usepackage{hyperref}
\usepackage[style=long]{glossaries-extra}
\usepackage{pst-node}
\usepackage{tikz-cd}
\usepackage{tikz}
\usepackage{caption}
\usepackage{subcaption}
\usepackage{float}
\usetikzlibrary{calc}
\usepackage{fancyhdr}
\usepackage{etoolbox}
\usepackage{pdfpages}
\usetikzlibrary{arrows}
\usepackage{nicefrac}
\usepackage{calc}
\usepackage{pst-solides3d}
\usepackage{pgfplots}
\usepgfplotslibrary{patchplots}
\usetikzlibrary{patterns, positioning, arrows}
\usetikzlibrary{fadings}
\usepackage[utf8]{inputenc}
\usepackage[T1]{fontenc}
\usepackage{lmodern}
\usepackage{comment}
\usepackage{enumitem}
\usepackage{authblk}
\usepackage{tikz-cd}
\usetikzlibrary{arrows}
\usepackage{mathtools}
\usepackage[all]{xy}
\usepackage[textsize=footnotesize,textwidth=4cm,backgroundcolor=orange!20,linecolor=orange,bordercolor=orange]{todonotes}

\theoremstyle{plain}
\newtheorem{theo}{Theorem}[section]
\newtheorem{prop}[theo]{Proposition}
\newtheorem{lem}[theo]{Lemma}
\newtheorem{coro}[theo]{Corollary}
\newtheorem*{conj}{Conjecture}

\theoremstyle{definition}
\newtheorem{df}[theo]{Definition}

\newtheorem{rmk}[theo]{Remark}

\newcommand{\Z}{\mathbb{Z}}

\newcommand{\calC}{\mathcal{C}}

\newcommand{\F}{\mathbb{F}}
\newcommand{\calG}{\mathcal{G}}
\newcommand{\calH}{\mathcal{H}}

\newcommand{\calU}{\mathcal{U}}

\newcommand{\cic}[1]{\left\langle #1 \right\rangle}
\newcommand{\til}[1]{\tilde{#1}}

\newcommand{\bl}{\backslash}

\DeclareMathOperator{\im}{im}
\DeclareMathOperator{\HNN}{HNN}

\providecommand{\keywords}[1]
{\small	\textbf{Keywords:} #1.}

\providecommand{\MSC}[1]
{\small	\textbf{MSC Classification:} #1.}

\makeatletter
\newcommand\blfootnote{\gdef\@thefnmark{}\@footnotetext}
\makeatother

\definecolor{mygreen}{rgb}{0, 0.3, 0}
\definecolor{myred}{rgb}{0.4,0,0}
\definecolor{myblue}{rgb}{0.1,0.1,0.6}

\hypersetup{
     colorlinks=true,
     linkcolor=mygreen,
     filecolor=myblue,
     citecolor = myred,      
     urlcolor=myblue,
}

\title{Pro-$\calC$ groups acting on profinite trees\blfootnote{Both authors were supported by CNPq.}}

\author[1]{\href{https://lcorrealopes.github.io/home/}{Lucas C. Lopes}}

\author[2]{\href{https://mat.unb.br/\~pz/index.html}{Pavel A. Zalesskii}}

\affil[1]{\begin{minipage}[t]{\textwidth}Universidade Federal de Minas Gerais, Departamento de Matemática, 31270-901 Belo Horizonte, Brazil, email address: \href{mailto:lucasclopes@ufmg.br}{lucasclopes@ufmg.br}.\end{minipage}}

\affil[2]{\begin{minipage}[t]{\textwidth}{Universidade de Brasília, Departamento de Matemática, 70910-900 Brasília, Brazil, email address: \href{mailto:pz@mat.unb.br}{pz@mat.unb.br}}.\end{minipage}}

\begin{document}

\maketitle

\begin{abstract}
We provide necessary conditions for pro-$\calC$ subgroups to embed into free profinite products, where $\calC$ is a variety of finite groups that does not contain all finite groups. Under suitable hypotheses, we extend this result to profinite groups acting $k$-acylindrically on profinite trees. Finally, we show that these results can be applied to important classes of profinite groups.
\end{abstract}

\MSC{20E08, 20E18, 20F64}

\keywords{free profinite products, fundamental groups, acylindrical actions, geometric group theory}

\section{Introduction}\label{sec0}

The Kurosh subgroup theorem describing the algebraic structure of subgroups of free products is one of the fundamental results of combinatorial group theory. In the profinite setting, however, the direct analogue of this theorem does not hold for free products of profinite groups. Consequently, the structure of closed subgroups of free profinite products is considerably more subtle and remains far from being fully understood.

The study of prosoluble subgroups of free profinite products was initiated in the middle $90$s by M. Jarden and F. Pop (see \cite{Jar94,Pop95}). More recently, the second author gave a description of finitely generated and second countable prosoluble subgroups of free profinite products, extending the main theorem of \cite{Pop95}. If $\calC$ is a variety of finite groups containing a non-soluble finite group and one considers pro-$\calC$ subgroups of free profinite products, the problem becomes significantly harder. Indeed, several of the main tools useful in the prosoluble case, such as cohomological criteria for relative projectivity, are no longer available. The main objective of this paper is to extend the study to  pro-$\calC$ subgroups, where $\calC$ is the class of finite groups closed under taking subgroups, quotient and extensions.

\begin{theo}\label{theo:kuroshtypedecomposition}
Let $G = \coprod_{i \in I} G_i$ be a free profinite product of profinite groups $G_i$. Let  $\calC$ be a class of finite groups closed under taking subgroups, quotient and extensions that does not contain all finite groups and $H$ a pro-$\calC$ subgroup of $G$. Then one of the following holds:
\begin{enumerate}[label = (\roman*)]
\item all non-trivial $H \cap G_i^g$ are pro-$\pi$ for $g \in G$, $i \in I$ where $\pi$ is a finite set of primes. Moreover, if $H$ is finitely generated then $ H \cap G_i^g$ are finitely generated and the number of non conjugate and non-trivial such intersections is bounded.

\item $H \leq G_i^g$ for some $g \in G$, $i \in I$.
\end{enumerate}
\end{theo}

As the profinite fundamental group of a finite graph of profinite groups with finite edge groups is virtually a free profinite product (see \cite{LZ24}), we can deduce the following:

\begin{coro}\label{theo:finite-edge-groups}
Let $(\calG,\Gamma)$ be an injective finite graph of profinite groups with finite edge groups, $G$ its profinite fundamental group. Suppose $\calC$ does not contain all finite groups. If $H$ is a pro-$\calC$ subgroup of $G$, then one of the following holds:
\begin{enumerate}[label = (\roman*)]
\item all non-trivial $H \cap G(v)^g$ are pro-$\pi$ for $g \in G$, where $\pi$ is a finite set of primes and $v \in V(\Gamma)$.

\item $H \leq G(v)^g$ for some $g \in G$ and $v \in V(\Gamma)$.
\end{enumerate}
\end{coro}

Theorem \ref{theo:kuroshtypedecomposition} applies to the profinite completion of the fundamental group $\pi_1(M)$ of a reducible $3$-manifold; namely the Kneser–Milnor decomposition of $\pi_1(M)$ gives the free profinite product decomposition of $\widehat{\pi_1(M)}$ so that Theorem \ref{theo:kuroshtypedecomposition} can be used.

\medskip
Note that many groups of geometric nature (like $3$-manifold groups, virtually special groups etc.) are the fundamental groups of acylindrical graphs of groups (a graph of groups is called $k$-acylindrical if its fundamental group acts on the Bass--Serre tree $k$-acylindrically, i.e., such that the intersection of $k+1$ edge stabilizers is always trivial). This was the motivation for extending the description of  prosoluble groups of free profinite products to the description of prosoluble subgroups of the fundamental groups of $k$-acylindrical graphs of profinite groups \cite{LZ24}. To obtain this characterization the cohomological criterion for relative projectivity from \cite{Zal23} (also proved independently in \cite{Wil25}) was essentially used. In the absence of such criteria for the non-soluble case we had to put some restrictions on the groups.

\begin{theo}\label{theo:pro-pi-stabilizers-profinite-version}
Let $(\calG,\Gamma)$ be a $k$-acylindrical finite graph of profinite groups, $G$ its fundamental group and $T$ its Bass--Serre tree. Suppose every infinite cyclic pro-$p$ group of $G$ has an open subgroup which is a virtual retract of $G$. Let $H$ be a finitely generated pro-$\calC$ subgroup of $G$ where $\calC$ does not contain all finite groups. If $H$ does not fix a vertex of $T$, then all $H$-vertex stabilizers are pro-$\pi$ groups for some finite set of primes $\pi$.
\end{theo}

The main application of Theorem \ref{theo:pro-pi-stabilizers-profinite-version} is the profinite completion of a hyperbolic or toral relatively hyperbolic virtually compact special group. Any such group $G$ has a subgroup $G_0$ of finite index that admits a malnormal quasiconvex hierarchy (see \cite{Dan21} and \cite{E}).  Since the edge groups in the hierarchy are malnormal, the corresponding action of $G_0$ on the Bass–Serre tree is $1$-acylindrical. As was proved in \cite{WZ17}, this malnormal hierarchy survives in the profinite completion so that the profinite completion $\widehat{G_0}$ acts $1$-acylindrically on the profinite tree and therefore we can apply Theorem \ref{theo:pro-pi-stabilizers-profinite-version}.

\begin{theo}\label{theo:virtually compact special case}
Let $G_0$ be a compact special group and $\widehat{G_0}$ be the profinite completion of $G_0$ acting acylindrically on a profinite tree $T$ as described above. If $H$ is a pro-$\calC$ subgroup of $\widehat{G_0}$ where $\calC$ does not contain all finite groups, then all $H$-stabilizers are pro-$\pi$ for a finite set of primes $\pi$ whenever $H$ does not fix a vertex.
\end{theo}

In particular,  Theorem \ref{theo:pro-pi-stabilizers-profinite-version} can be applied to the profinite completion of well known groups of geometric nature. Namely:

\begin{coro}\label{coro:application}
The hypotheses of Theorem \ref{theo:pro-pi-stabilizers-profinite-version} apply to the profinite completion of the following groups:
\begin{enumerate}
\item[(i)] limit groups. In particular $\pi_1(S_g)$, where $S_g$ is a closed orientable surface of genus $g\geq 2$,

\item[(ii)] fundamental group of hyperbolic $3$-manifolds,

\item[(iii)] finitely generated Fuchsian non-triangle groups.
\end{enumerate}
\end{coro}

Theorem \ref{theo:virtually compact special case} applies also to the following groups:

\begin{coro}\label{coro:application2}
The hypotheses of Theorem \ref{theo:virtually compact special case} apply, up to passing to an open subgroup, to the profinite completion of the following groups:
\begin{enumerate}
\item[(i)] word-hyperbolic graph braid groups,

\item[(ii)] finitely generated Kleinian groups,

\item[(iii)] one-relator groups with torsion.

\item[(iv)] Hyperbolic hyperbolic-by-cyclic groups, i.e., a semidirect product $G\rtimes\Z$ of a hyperbolic group $G$ with infinite cyclic group,

\item[(v)] the fundamental group of standard arithmetic hyperbolic manifolds.
\end{enumerate}
\end{coro}

This paper is organized as follows. In Section \ref{sec1} we introduce the basic ideas and notations of profinite Bass--Serre theory used in the paper. In Section \ref{sec2} we consider relatively projective profinite groups and prove the second part of the statement (i) of Theorem \ref{theo:kuroshtypedecomposition}. Section \ref{sec3} is dedicated to proving Theorem \ref{theo:kuroshtypedecomposition}. Section \ref{sec4} proves Theorem \ref{theo:pro-pi-stabilizers-profinite-version}, Theorem \ref{theo:virtually compact special case}, Corollary \ref{coro:application} and Corollary \ref{coro:application2}.

Our subgroups are assumed to be closed, homomorphisms to be continuous, finitely generated means topologically finitely generated. Throughout the paper we denote by $\calC$ a variety closed under extensions according to the terminology of \cite{RZ10}.

\section{Preliminaries}\label{sec1}

This section introduces the main definitions of the profinite version of the Bass--Serre theory used throughout this work.

\begin{df}
A {\it profinite graph} is a profinite space $\Gamma$ with a distinguished non-empty subset $V(\Gamma)$ and two continuous maps $d_i: \Gamma \to V(\Gamma)$, $i = 0,1$, which are identity on $V$.
\end{df}

The subset $V(\Gamma)$ is the {\it vertex-set} of $\Gamma$ and $E(\Gamma) = \Gamma - V$ is the {\it edge-set} of $\Gamma$. We call $d_0(e)$ and $d_1(e)$, respectively, by {\it initial} and {\it terminal} vertices. The maps $d_0,d_1$ are the {\it incidence maps}. A profinite graph is connected if all finite quotients of it are connected as usual graphs.

\begin{df}\label{simply connected}
The pro-$\calC$ fundamental group of a finite  connected graph $\Gamma$ is the pro-$\calC$ completion  of the usual fundamental group of $\Gamma$. If $\Gamma$ is a profinite connected graph, one decomposes it as inverse limit $\Gamma = \varprojlim (\Gamma_i)$ of finite graphs $\Gamma_i$. 
The pro-$\calC$ fundamental group of  $\Gamma$ is defined to be
$$\pi_1^{\calC}(\Gamma) = \varprojlim \pi_1^{\calC}(\Gamma_i).$$
It does not depend on the choice of the decomposition (see \cite{Rib17} for details).
If $\calC$ contains all finite groups, then we denote it by $\pi_1(\Gamma)$. We say that a profinite graph $\Gamma$ is {\it $\calC$-simply connected} if $\pi_1^{\calC}(\Gamma) = 1$
\end{df}

Let $\Gamma$ be a profinite graph. We define
$$E^{\ast}(\Gamma) = \Gamma/V(\Gamma).$$
Let $R$ be a profinite ring and $\pi(\calC)$ the set of primes involved in $\calC$. Denote by $C(\Gamma,R)$ the chain complex
$$0 \longrightarrow{} [\![R(E^{\ast}(\Gamma),\ast)]\!]\buildrel{d}\over \longrightarrow [\![R(V(\Gamma))]\!] \buildrel{\varepsilon}\over\longrightarrow R \longrightarrow{} 0$$ 
where $\varepsilon(v) = 1$, $d(\bar{e}) = d_1(e) - d_0(e)$ for the image $\bar{e}$ of $e \in E(\Gamma)$ in $E^{\ast}(\Gamma)$ and $d(\ast) = 0$. The homology groups of $\Gamma$ are defined by
$$H_0(\Gamma,R) = \ker(\varepsilon)/\im(d),\quad H_1(\Gamma,R) = \ker(d).$$

\begin{df}
A profinite graph $\Gamma$ is a {\it pro-$\calC$ tree} if the sequence $C(\Gamma,\F_p)$ is exact for every $p \in \pi(\calC)$. 
\end{df}

Note that, in particular, $\Gamma$ is connected if, and only if, $H_0(\Gamma,R) = 0$ for any profinite ring $R$.

Also, if we define
$$T^g = \{m \in T: gm = m\}$$
to be the set of $g$-fixed points for a pro-$\calC$ tree $T$, then $T^g$ is again a pro-$\calC$ tree (see \cite[Theorem 4.1.5]{Rib17}).

\begin{df}
A {\it finite graph of pro-$\calC$ groups} is a pair $(\calG,\Gamma)$ where $\Gamma$ is a connected finite graph, $\calG$ consists of pro-$\calC$ groups $\calG(m)$ for each $m\in \Gamma$ and monomorphisms $\partial_i: \calG(e) \to \calG(d_i(e))$, for $i=0,1$, for each edge $e \in E(\Gamma)$.
\end{df}

\begin{df}
Let $(\calG,\Gamma)$ be a finite graph of pro-$\calC$ groups and $T$ be a maximal subtree of $\Gamma$. Then the {\it fundamental pro-$\calC$ group} of $(\calG,\Gamma)$ is defined by
$$\Pi_1(\calG,\Gamma) = \Pi_1^{\calC}(\calG,\Gamma) = \bigg(F\amalg \coprod_{v \in V(\Gamma)} \calG(v)\bigg)/N$$
where $F$ is the free pro-$\calC$ group with basis $\{t_e : e \in E(\Gamma)\}$ and $N$ is the closed normal closure of the set
$$\left\{t_e : e \in E(T)\right\} \cup \left\{\partial_0(x)^{-1}t_e\partial_1(x)t_e^{-1} : x \in \calG(e), e \in E(\Gamma)\right\}.$$
Note that the image of $t_e\in \Pi_1(\calG,\Gamma)$ is trivial when $e\in E(T)$.
\end{df}

\begin{rmk}
Unlike in the abstract case, a vertex group may not embed into the fundamental pro-$\calC$ group. When this holds, we will say that the finite graph of pro-$\calC$ groups is {\it injective}. Replacing the vertex groups with their images, we can construct another finite graph of pro-$\calC$ groups (over the same original graph) which is injective and whose fundamental pro-$\calC$ group is isomorphic to the original (see \cite[Section 6.4]{Rib17}). So there is no loss of generality in assuming that our graphs of groups are always injective.
\end{rmk}

The basic free constructions such as free profinite products, free profinite products with amalgamation and $\HNN$-extensions are the most basic examples of fundamental groups of finite graphs of groups (see \cite[Section 6.2]{Rib17}).

\begin{df}
Given an injective finite graph of pro-$\calC$ groups $(\calG,\Gamma)$ with fundamental group $G$, the corresponding {\it standard graph} is the disjoint union
$$S = S(G) = \bigcup_{m \in \Gamma} G/\calG(m).$$
The vertices of $S$ are the cosets $g\calG(v)$ with $v \in V(\Gamma)$ and $g \in G$; the edges of $S$ are the cosets $g\calG(e)$ with $e \in E(\Gamma)$ and $g \in G$; the incidence maps of $S$ are given by
$$d_0(g\calG(e)) = g\calG(d_0(e)),\quad d_1(g\calG(e)) = gt_e\calG(d_1(e))$$
with $e \in E(\Gamma)$ and $t_e = 1$ if $e \in E(T)$.
\end{df}

The fact that $S$ is indeed a profinite graph as well as the basic properties of $S$ are explained in \cite[Section 6.3]{Rib17}. In particular, it is proved in \cite[Theorem 6.3.5]{Rib17} that $S$ is $\calC$-simply connected and so is a pro-$\calC$ tree. We refer to $S$ as {\it Bass--Serre tree} following a terminology commonly used for the abstract version of Bass--Serre theory.

\begin{df}
An action of a profinite group $G$ on a profinite graph $T$ is called {\it $k$-acylindrical} if the stabilizer of any geodesic $[v,w]$ (i.e., the intersection of all subtrees containing $v$ and $w$) in $T$ of length greater than $k$ is trivial. We say that a profinite graph of profinite groups $(\calG,\Gamma)$ is {\it $k$-acylindrical} if its fundamental group acts $k$-acylindrically on its standard pro-$\calC$ tree.
\end{df}

\section{Relatively projective profinite groups}\label{sec2}

In this section we will show that parabolic subgroups of a pro-$\calC$ group which is profinite relatively projective are bounded if the variety $\calC$ does not contain all finite groups.

\begin{df}
A {\it group-pile} is a pair $\mathbf{G} = (G,T)$ consisting of a profinite group $G$, a profinite space $T$ and a continuous action of $G$ on $T$. The stabilizer of $t \in T$ will be denoted by $G_t$.
\end{df}

\begin{df}
A {\it morphism} of group-piles $\alpha : (G,T) \to (G',T')$ consists of a group homomorphism $\alpha : G \to G'$ and a continuous map $\alpha : T \to T'$ such that $\alpha(tg) = \alpha(t)\alpha(g)$ for all $t \in T$ and $g \in G$. The above morphism is an {\it epimorphism} if $\alpha(G) = G'$, $\alpha(T) = T'$ and for every $t' \in T'$, there is $t \in T$ such that $\alpha(t) = t'$ and $\alpha(G_t) = G'_{t'}$. It is {\it rigid} if $\alpha$ maps $G_t$ isomorphically onto $G'_{\alpha(t)}$, for every $t \in T$, and the induced map of the orbit spaces $T/G \to T'/G'$ is a homeomorphism.
\end{df}

\begin{df}
Let $(G,T)$ be a group-pile. An {\it embedding problem} for $(G,T)$ is a pair of morphisms
$$\begin{tikzpicture}[auto]
\node(X) at (0,0) {$(B,T_B)$};
\node(G) at (3,0) {$(A,T_A)$};
\node(F) at (3,2) {$(G,T)$};
\path[->] (X) edge [] node[below] {$\alpha$} (G);
\path[<-] (G) edge [] node[right] {$\varphi$} (F);
\end{tikzpicture}$$
such that $\alpha$ is an epimorphism. It is a {\it $\calC$-embedding problem} if $G,B,A$ are pro-$\calC$. The embedding problem is {\it finite} if $B$ is finite and is {\it rigid} if $\alpha$ is rigid. A {\it solution} to the embedding problem is a morphism $\gamma:(G,T) \to (B,T_B)$ such that $\alpha \gamma = \varphi$.
\end{df}

\begin{df}
We say that $(G,T)$ is {\it $\calC$-projective} if every rigid finite $\calC$-embedding problem for $(G,T)$ has a solution. In this case we say that $G$ is {\it $\calC$-projective relative to} $\{G_t : t \in T\}$. Following the tradition of geometric group theory we shall call $G_t$ parabolic subgroups.
\end{df}

This definition of projectivity is equivalent to the definition of strongly projectivity of \cite{HJ21}.

\medskip

We shall need  the following results that will be used throughout this paper.

\begin{prop}{\cite[Proposition 2.5]{Zal24}}\label{prop:embedding}
	Let $\calC$ be a class of finite groups closed under taking subgroups, quotients and extensions. Let $G$ be $\calC$-projective relative to
	$$\calG = \{G_t : t \in T\}$$
	such that there exists a continuous section $\sigma: G \bl T \to T$. Then $G$ embeds into a free pro-$\calC$ product
	$$\coprod_{s \in \im(\sigma)} G_s \amalg F$$
	where $F$ is a free pro-$\calC$ group of rank $d(G)$, the minimal number of generators of $G$. Moreover, every $G_t$ is of the form $G_s^g \cap G$ for some $g \in \coprod_{s \in \im(\sigma)} G_s \amalg F$, $s \in \im(\sigma).$
\end{prop}

\begin{prop}\label{subgroups} \cite[Proposition 5.4.2]{HJ21}
	Let $G$ be $\calC$-projective relative to $\calG= \{G_t : t \in T\}$ and $H$ a subgroup of $G$. Then $H$ is $\calC$-projective relative to $\{H \cap \Gamma : \Gamma \in \calG^G\}$.
\end{prop}

We have the following result of independent interest (it generalizes \cite[Proposition 4.7]{Zal24}); we denote by $\alpha(G)$ the number of finite simple $\pi$-groups $S$ in the quotient $G/M(G) \simeq \prod S$, where $M(G)$ is the intersection of all maximal normal subgroups of $G$ and $\pi$ a set of primes. (cf. \cite[Section 8.2]{RZ10}).

\begin{prop}\label{relproj}
Let $(G,T)$ be a $\calC$-projective group-pile and $T_0 = \{t \in T : G_t \neq 1\}$. If $G$ is finitely generated and $G_t$, $t \in T$ are all pro-$\pi$, for a finite set of primes $\pi$, then $T_0$ is closed in $T$ and $\sum_t d(G_t) \leq \alpha(G) d(G)$, where $t$ runs through representatives of $G$-orbits in $T_0$.
\end{prop}

\begin{proof}
Suppose that $\sum_t d(G_t) > \alpha(G) d(G)$. Then there exists an epimorphism $\varphi: G \to A$ to a finite group $A$ such that there are non-trivial non-conjugate subgroups $G_{t_1}$,...,$G_{t_n}$ with images $A_1 = \varphi(G_{t_1})$,..., $A_n = \varphi(G_{t_n})$, none of which is contained in a conjugate of the other and with $\varphi\left(\bigcup_{t \in T} G_t\right) = \bigcup_1^n A_i$ such that $\sum_i d(A_i) > \alpha(G) d(G)$.

Define $T_A = \{\varphi(G_t) : t \in T\}$. Let $f: F_{\calC} \to A$ be an epimorphism from a free pro-$\calC$ group $F_{\calC}$ and $B = \coprod_1^n A_i \amalg F_{\calC}$ a free pro-$\calC$ product. Put $T_B = \{A_i^b : b \in B\}$. Let $\beta:(B,T_B) \to (A,T_A)$ be a rigid morphism of group-piles defined by sending $A_i$ to their copies in $A$. Since $(G,T)$ is $\calC$-projective, the embedding problem
$$\begin{tikzpicture}[auto]
\node(X) at (0,0) {$(B,T_B)$};
\node(G) at (3,0) {$(A,T_A)$};
\node(F) at (3,2) {$(G,T)$};
\path[dashed,<-] (X) edge [] node[above left] {$\gamma$} (F);
\path[->] (X) edge [] node[below] {$\beta$} (G);
\path[->] (G) edge [] node[right] {$\varphi$} (F);
\end{tikzpicture}$$
admits a solution $\gamma: (G,T) \to (B,T_B)$ (see \cite[Lemma 5.3.1]{HJ21}). Note that $\gamma(G_{t_i})$ is conjugate to $A_i$. So the natural epimorphism $B \to \prod_i^n A_i$ restricts surjectively on $\gamma(G)$. Taking the quotient of each $A_i$ by a maximal normal subgroup, we obtain an epimorphism from $\gamma(G)$ onto a direct product of $n$ finite simple $\pi$-groups. Since $\gamma(G)$ is a quotient of $G$, the definition of $\alpha(G)$ ensures $n \leq \alpha(G)$ so the number of $A_i$ is bounded by $\alpha(G)$ as well. Choose $j$ such that $d(A_j) \geq d(A_i)$ for $i \neq j$ and note that
$$\alpha(G) d(A_j) \geq n d(A_j) \geq \sum_i d(A_i) > \alpha(G) d(G),$$
that is, $d(A_j) > d(G)$, a contradiction since there is an epimorphism from $\gamma(G)$ to $A_j$.
\end{proof}

\section{Subgroups of free profinite products}\label{sec3}

In this section we will prove our first main results about pro-$\calC$ subgroups of free profinite products.

\begin{df}
The group $G$ is said to be {\it invariably generated} by elements $a,b\in G$ if $G=\langle a^g, b^h\rangle$ for all $g,h\in G$.
\end{df}

We shall use the following result proved by M. Jarden in \cite[Lemma 1.1-1.3]{Jar94}. 

\begin{lem}\label{invgen}(\cite{Jar94})
 Let $p < q$ be primes. Then $S_q$ is invariably generated by element of order $2$ and an element of order $q$ and  $A_q$ is invariably generated by an element of order $p < q - 2$ and an element of order $q$.
\end{lem}

\begin{prop}\label{prop:nonconjugated}
Let $G = \coprod_1^n G_i$ be a free profinite product of profinite groups $G_i$. Let $H$ be a subgroup of $G$ such that there exist some non-trivial subgroups $H_i = H \cap G_i^{g_1}$, $H_j = H \cap G_j^{g_2}$ which are not both pro-$\pi$ for the same finite set of primes $\pi$. Then $H$ is not pro-$\calC$ if $\calC$ does not contain all finite groups.
\end{prop}

\begin{proof}
If $H_i$ and $H_j$ are not both pro-$\pi$, then we can assume that infinitely many primes divide $|H_j|$, otherwise take $\sigma$ as the union of all primes in $|H_i|$ and in $|H_j|$ such that both would be pro-$\sigma$ for a finite set of primes $\sigma$. If both are pro-$\pi$ for an infinite set of primes $\pi$, then we can also assume that infinitely many primes divide $|H_j|$. Then, fixing some $p$ dividing $|H_i|$ we have infinitely many pairs of primes $(p,q)$ such that $q$ divide $|H_j|$. Fix a pair $(p,q)$ with $p < q - 2$.

Taking $N_i$ as an open normal subgroup of $G_i$ which, up to conjugation, does not contain the $p$-Sylow of $H_i$ and $N_j$ correspondingly, we can set $\bar{G_i} = G_i/N_i$, $\bar{G_j} = G_j/N_j$ and $\bar{G} = \bar{G_i} \amalg \bar{G_j}$. The image $\bar{H}$ of $H$ under the canonical projection $\coprod G_i \to \bar{G}$ is such that the image of $H_i$ contains a subgroup of order $p$, the image of $H_j$ contains a subgroup of order $q$. Since pro-$\calC$ groups are closed under continuous images and closed
subgroups, to show that $H$ is not pro-$\calC$ it is enough, for each fixed pair $(p,q)$, to find a subgroup of $\bar{H}$ having a finite quotient isomorphic to $S_q$ or $A_q$ because every finite group
embeds into some $S_n$ and $S_n$ embeds into $A_{n+2}$. 

Thus we may assume $G = \coprod G_i = G_i \amalg G_j$ and that $G_i$ and $G_j$ are finite groups. We can choose (and fix) copies of $C_p$ and of $C_q$ in $H_i$ and $H_j$, respectively. Let $\pi: G \to G_i \times G_j$ be the canonical projection. The images of the chosen copies of $C_p$ and of $C_q$ under $\pi$ generate a subgroup of $\pi(H)$ isomorphic to $C_p \times C_q$, which again we denote by $C_p \times C_q$. Let $U \subset G$ be the inverse image of $C_p \times C_q$ by $\pi$ and set $\til{U} = H \cap U$. By construction, $\til{U}$ contains the subgroup of $H$ generated by the chosen copies of $C_p$ and $C_q$. Since $U$ is open, we can apply the profinite version of the Kurosh subgroup theorem to split $U$ as a free profinite product where the chosen copies of $C_p$ and $C_q$ are conjugate in $U$ to some free factors $U_{s_i}$, $U_{s_j}$, respectively. Thus we obtain an epimorphism $w: U \to U_{s_i} \amalg U_{s_j}$ that sends $U_{s_i}$ onto $U_{s_i}$ and $U_{s_j}$ onto $U_{s_j}$ identically (cf. \cite[Proof of Claim 1]{Pop95}). Since the chosen copies of $C_p$ and $C_q$ are in $\til{U}$ and are conjugate to $U_{s_i}$, $U_{s_j}$, their images under $w$ are, up to conjugation, $U_{s_i}$ and $U_{s_j}$ so that they are in $w(\til{U})$. In particular, $w(\til{U})$ contains elements of orders $p$ and $q$ which are conjugate to generators of the two factors. Then  by  Lemma \ref{invgen} there exists either an epimorphism $U_{s_i} \amalg U_{s_j} \to S_q$ or $U_{s_i} \amalg U_{s_j} \to A_q$ that restricts surjectively to $w(\til{U})$. This finishes the proof.
\end{proof}

From Theorem \ref{theo:kuroshtypedecomposition}, we also can deduce the following:

\begin{coro}\label{finvs}
Let $G = \coprod_{i \in I} G_i$ be a free profinite product of profinite groups $G_i$. Suppose that $\calC$ does not contain all finite groups. If $H$ is a finitely generated pro-$\calC$ subgroup of $G$, then all $H \cap G_i^g$ are finitely generated and there are   only finitely many $H \cap G_i^g\neq 1$ up to conjugation.
\end{coro}

\begin{proof} If $H$ is finitely generated, then both statements follows from Proposition \ref{relproj} taking into account Proposition \ref{subgroups}.
\end{proof}

\section{Subgroups of profinite groups acting on trees}\label{sec4}

In this section we shall prove Corollary \ref{theo:finite-edge-groups}, Theorem \ref{theo:pro-pi-stabilizers-profinite-version} and Theorem \ref{theo:virtually compact special case}.

From Theorem \ref{theo:kuroshtypedecomposition}) we can deduce Corollary \ref{theo:finite-edge-groups}:

\begin{proof}[Proof of Corollary \ref{theo:finite-edge-groups}]
Since $\Gamma$ is finite and all edge groups are finite, there exists open normal subgroup $U$ such that $U \cap G(e)^g = 1$ for all $e$. By \cite[Corollary 4.5]{ZM89}, $U$ splits as a free
profinite product $\coprod_i U_i$ where $U_i$ are of the form
$U \cap G(v)^g$, $v \in V(\Gamma)$. Set $H_0 = H \cap U$. If $H_0$ fixes a vertex, $S^{H_0}$ is a nonempty $H$-invariant profinite tree and $H/H_0$ fixes a vertex of $S^{H_0}$ since it is finite, then $H$ fixes a vertex as well. If $H_0$ does not fix a vertex, then every $H_0 \cap U_i^g$ is pro-$\pi$ for a finite set of primes $\pi$ (cf. \cite[Theorem 4.2]{LZ24}). Since $U \cap G(v)^g$ is a $U$-vertex stabilizer, we have that
$$H_0 \cap G(v)^g = H_0 \cap (U \cap G(v)^g)$$
so that $H_0 \cap G(v)^g$ is pro-$\pi$. Note that the index of $H_0 \cap G(v)^g$ in $H \cap G(v)^g$ does not exceed $|G:U|$. It finishes the proof replacing $\pi$ by a possibly larger, but still finite, set of primes.
\end{proof}

We will prove two technical results. First, we prove a natural (partial) extension of \cite[Theorem 4.8]{LZ24}. If $G$ is a pro-$\calC$ acting on a profinite tree $T$ and $U$ a subgroup of $G$ we denote by $\widetilde{U}$ the group generated by all intersections of $U$ with vertex stabilizers of $G$.

\begin{lem}\label{lem:ontocpcq}
Let $G$ be a pro-$\calC$ group acting $k$-acylindrically on a profinite tree $T$. Suppose for some primes $p\neq q$ and some $v,w\in V(T)$ from different $G$-orbits there is a subgroup $K$ of $G$ such that $G_v \cap K$ has an epimorphism onto $C_p$ and $G_w \cap K$ has an epimorphism onto $C_q$. Then there exists a subgroup $H$ of $G$ that admits an epimorphism $f:H \to C_p \times C_q$ such that $|f(H_v)|=p, |f(H_w)|=q$, $\ker(f)$ is not generated by its vertex stabilizers.
\end{lem}

\begin{proof}
W.l.o.g. we may assume that $G_v$ and $G_w$ are maximal (by inclusion) vertex stabilizers. Choose an open normal subgroup $U$ of $G$ such that $G_vU/U$ and $G_wU/U$ have subgroups $G_p$ of order $p$ and $G_q$ of order $q$, respectively. Consider the action of $G_U=G/\widetilde U$ on $T_U=\widetilde U \bl S$ and note that in $G_U$ the vertex stabilizers of the images $v_u, w_u$ of $v$ and $w$ in $T_U$ are  $G_vU/U$ and $G_wU/U$ (still in different orbits) and refining $U$ we may assume that $G_vU/U$ and $G_wU/U$ are maximal vertex stabilizers. By \cite[Proposition 3.9]{LZ24} $G_U$ acts on a simply connected (infinite) profinite tree $T^{\ast}_U$ with trivial edge stabilizers, the vertex stabilizers of $T^{\ast}_U$ are maximal vertex stabilizers of $T_U$. Hence $G_p$ and $G_q$ are subgroups of some  stabilizers of vertices of $T^{\ast}_U$ from distinct $G_U$-orbits. Note that if a subgroup of $G$ is generated by its vertex stabilizers then its image in $G_U$ is generated by its vertex stabilizers as well, and vertex stabilizers in $G_U$ are finite. Therefore to prove the lemma it suffices to consider $H$ in $G_U$ containing $G_p$ and $G_q$ and construct an epimorphism $f:H \to C_p \times C_q$, where $C_p$ and $C_q$ are cyclic groups of order $p$ and $q$, with $|f(G_p)|=p, |f(G_q)|=q$ such that $\ker(f)$ is not generated by torsion. For simplicity of notation we will consider, under the reduction to $G_U$, $H = \cic{G_v,G_w}$. If there is no epimorphism from $G_v$ onto $C_p$ and from $G_w$ onto $C_q$ we can replace $G_v$ by $G_v \cap K$ and $G_w$ by $G_w \cap K$. In this case we would have $H = \cic{G_v \cap K, G_w \cap K}$ with $H_v = G_v \cap K$ and $H_w = G_w \cap K$. So there is no loss of generality in assuming that there are epimorphisms $f_v: G_v \to C_p$ and $f_w: G_w \to C_q$.

Using the Mayer--Vietoris sequence for groups acting on trees (see \cite{Mel90}), we have
$$\cdots \to 0 \to \bigoplus_{x \in H \bl V(T)} H_1(H_x) \to H_1(H) \to \bigoplus_{y \in H \bl E(T)} H_0(H_y) \stackrel{d_0}{\to} \bigoplus_{x \in H \bl V(T)} H_0(H_x) \to \cdots$$
where we are simplifying $H_i(-,\widehat{\Z})$ as $H_i(-)$. Set $Y = H \bl T$. Note that $Y$ is a tree because $H$ is generated by vertex stabilizers (see \cite[Proposition 3.9.2]{Rib17}). We have the chain complex associated to $Y$ and the map $d: [[\widehat{\Z}(E*(Y),*)]] \to [[\widehat{\Z}V(Y)]]$. Since $Y$ is a tree, the first homology group $H_1(Y,\widehat{\Z})$ is trivial, that is, $\ker(d) = 0$. Now the maps $d_0$ and $d$ are the same so that $d_0$ is injective. Since $d_0$ is injective, we have the isomorphism
$$\bigoplus_{x \in H \bl V(T)} H_1(H_x) \simeq H_1(H).$$
Recall that $A^{ab} = A/[A,A]$ for any group $A$. Since $H$ is generated by $G_v$ and $G_w$ then $H^{ab}$ is generated by $G_v^{ab}$ and $G_w^{ab}$ since the abelianization is a continuous epimorphism. So, we can reduce our sum to
$$H^{ab} \simeq G_v^{ab} \oplus G_w^{ab}.$$
Since there is an epimorphism $f_v: G_v \to C_p$, then $[G_v,G_v] \leq \ker f_v$, hence, it induces a well-defined epimorphism $\til{f_v}: G_v^{ab} \to C_p$. In the same way there is an epimorphism $\til{f_w}: G_w^{ab} \to C_q$. We can define an epimorphism $\varphi: G_v^{ab} \oplus G_w^{ab} \to C_p \oplus C_q$ by $\varphi(x,y) = (\til{f_v}(x),\til{f_w}(y))$. Thus composing $\varphi: H^{ab} \to C_p \oplus C_q$ with $\pi: H \to H^{ab}$ we obtain an epimorphism $f: H \to C_p \oplus C_q$. By construction, $|f(H_v)| = p$ and $|f(H_w)| = q$.

Note that all $G$-stabilizers except for stabilizers of $v$ and $w$, up to translation, vanish in the abelianization of $H$. Let $L$ be the group generated by these stabilizers and consider the quotient $H/L$, which is non-abelian since it has distinct non-trivial stabilizers. The image of $\widetilde{\ker(f)}$ in $H/L$ is trivial because an element of $C_p$ or $C_q$ cannot be in $\ker(f)$, otherwise $C_p$ and $C_q$ would be mapped onto $1$ by $f$. On the other hand, if we take a generator $a$ of $C_p$ and a generator $b$ of $C_q$, their commutator $[a,b]$ is in $\ker(f)$. Moreover, since $G_v \cap G_w = 1$ and $a$ and $b$ do not commute, the image of $\ker(f)$ in $H/L$ is non-trivial so that $\widetilde{\ker(f)} \neq \ker(f)$, which means that $\ker(f)$ is not generated by its own vertex stabilizers.

Thus, we have the desired epimorphism $f: H \to C_p \times C_q$.
\end{proof}

\begin{prop}\label{prop:nonconjugated2}
Let $(\calG,\Gamma)$ be a $k$-acylindrical finite graph of profinite groups, $G$ its fundamental group and $T$ its Bass--Serre tree. Let $H$ be a subgroup of $G$ such that there exist some non-trivial subgroups $H_v$, $H_w$ which are not both pro-$\pi$ for the same finite set of primes $\pi$ and $v \neq w \in T$. Assume that for primes $p \neq q$ with $C_p$ a subquotient of $H_v$ and $C_q$ a subquotient of $H_w$, there is a subgroup $U_H$ of $H$ such that $H_v \cap U_H$ has an epimorphism onto $C_p$ and $H_w \cap U_H$ has an epimorphism onto $C_q$. Then $H$ is not pro-$\calC$ if $\calC$ does not contain all finite groups. Moreover, if the set $\pi$ of primes $p$ such that $C_p$ is a subquotient of $H_v$ for some $v$ is infinite and, for every pair of distinct primes $p,q \in \pi$, there are vertices $v \neq w \in T$ and a subgroup $U_H \leq H$ such that $H_v \cap U_H$ has an epimorphism onto $C_p$ and $H_w\cap U_H$ has an epimorphism onto $C_q$, then the same conclusion holds.
\end{prop}

\begin{proof}
If $H_v$ and $H_w$ are not both pro-$\pi$, then we can assume that infinitely many primes divide $|H_w|$, otherwise take $\sigma$ as the union of all primes in $|H_v|$ and in $|H_w|$ such that both would be pro-$\sigma$ for a finite set of primes $\sigma$. If both are pro-$\pi$ for an infinite set of primes $\pi$, then we can also assume that infinitely many primes divide $|H_w|$. Then, fixing some $p$ dividing $|H_v|$ we have infinitely many pairs of primes $(p,q)$ such that $q$ divide $|H_w|$.

Fix a pair $(p,q)$ with $p \neq q$. We can apply Lemma \ref{lem:ontocpcq} to obtain an epimorphism $f_L: L \to C_p \times C_q$ from an open subgroup $L$ of $H$ with the kernel not generated by vertex stabilizers such that $f_L(L_v)$ has order $p$ and $f_L(L_w)$ has order $q$. We can extend this epimorphism to an epimorphism $f_U$ of some open subgroup $U$ of $G$ containing $L$ (see \cite[Proposition 8.3.8]{RZ10}). By \cite[Lemma 4.6]{LZ24} we can choose such $U$ with the kernel $K$ of $f_U$ not generated by its vertex stabilizers. Let $\widetilde K$ be the normal subgroup of $K$ generated by the vertex stabilizers. Then $\widetilde K$ is normal in $U$ and we can consider $U/\widetilde K$  acting on $\widetilde K\backslash T$ which is simply connected by \cite[Lemma 4.7]{ZM89}. Then by \cite[Proposition 4.4]{ZM89} $U/\widetilde K$ is the fundamental profinite group of a finite graph of finite groups $\Pi_1(\calU,U\backslash T)$ and so is the profinite completion of the abstract fundamental group $\Pi=\Pi_1^{abs}(\calU, U\backslash T)$ (cf. \cite[Proposition 6.5.6]{Rib17}). Note that we have an induced epimorphism $f_K:U/\widetilde K\to C_p \times C_q$ and we denote by $f^{abs}$ its restriction to $\Pi$. Since $K/\widetilde K$ is projective (see \cite[Corollary 4.1.3]{Rib17} or \cite[Theorem 2.6]{ZM88}) and so is torsion-free, the kernel $\ker(f^{abs})$ is free and so by \cite[Lemma 4.7]{LZ24} $f^{abs}$ factors through the free product $C_p \ast C_q$. Hence $f_K$ factors through the free profinite product $C_p\amalg C_q=\widehat{C_p \ast C_q}$. Thus we have the following commutative diagram:
 
$$\xymatrix{&& U\ar[d]&\\
           L\ar[rrdd]_{f_L}\ar[urr]\ar@{-->}[rrd]\ar[rr] &&U/\widetilde K\ar[d]&\Pi\ar[l]\ar[d]\\
           &&C_p\amalg C_q\ar[d]&C_p*C_q\ar[l]\ar[ld]\\
                        && C_p \times C_q&}$$          
 
Since $f_L(L_v)$ has order $p$ and $f_L(L_w)$ has order $q$ and the images of $L_v$ and $L_w$ in $U/\widetilde K$ are finite, it follows from this commutative diagram that the image of $L$ in $C_p \amalg C_q$ contains a group of order $p$ and a group of order $q$.

Now, as at the end of the proof of Proposition \ref{prop:nonconjugated}, we deduce from Lemma \ref{invgen} that $L$ has $S_q$ or $A_q$ as its quotients.

This argument holds for every pair $(p,q)$ satisfying Lemma \ref{invgen}. Recall that every finite group is a subgroup of $S_n$ for some natural $n$ and we can embed $S_n$ into $A_{n+2}$. Then $\calC$ must contain all finite groups, a contradiction.

The moreover part follows by exactly the same argument: fix one prime $p \in \pi$ and let $q\in \pi$ vary over infinitely many primes with $p<q$ and $q\neq 3$. For each such pair $(p,q)$, the hypotheses give vertices $v \neq w$ and a subgroup $U_H \leq H$ to which the preceding argument applies. Hence $\calC$ contains $A_q$ or $S_q$ for infinitely many $q$ and therefore contains all finite groups.
\end{proof}

Now we can prove the next result:

\begin{proof}[Proof of Theorem \ref{theo:pro-pi-stabilizers-profinite-version}]
Choose $A \simeq \Z_p \leq H_v$ and $B \simeq \Z_q \leq H_w$. Since open subgroups of infinite cyclic pro-$p$ groups are infinite cyclic pro-$p$, by hypothesis we may assume that we have retractions $r_p: U_p \to A$ and $r_q: U_q \to B$ where $U_p, U_q$ are open in $G$. Set $G_0 = U_p \cap U_q$, $A_0 = A \cap G_0$ and $B_0 = B \cap G_0$. Restricting $r_p$ to $G_0$ we obtain a homomorphism $G_0 \to A$. Since $r_p$ is a retraction, $A_0 \leq r_p(G_0)$ so that $r_p(G_0)$ is an infinite pro-$p$ group. We can apply the same argument to $r_q$. Setting $H_0 = H \cap G_0$, we have that the image of $(H_0)_v$ under $r_p$ and the image of $(H_0)_w$ under $r_q$ are infinite. Since $H_0$ is open in $H$, we may assume that $H = H_0$.

Let
$$\pi= \{p : C_p\text{ is a subquotient of some }H\text{-vertex stabilizer}\}.$$
If $\pi$ is infinite, then the construction above holds for every pair of distinct primes $p,q \in \pi$. Thus, by Proposition \ref{prop:nonconjugated2}, $H$ is not pro-$\calC$, a contradiction.
\end{proof}

\begin{df}
A {\it hierarchy of groups} of length $0$ is a single vertex labeled by a group. A {\it hierarchy of groups} of length $n$ is a graph of groups $(\calG_n,\Gamma_n)$ together with hierarchies of length $n-1$ on each vertex of $\Gamma_n$. If $\calH$ is a length $n$ hierarchy of groups, the $n$-th level of $\calH$ is the graph of groups $(\calG_n,\Gamma_n)$. For $1 \leq k \leq n$, the $(n-k)$-level of $\calH$ is the disjoint union of the $(n-k)$-th levels of the hierarchies on the vertices of $\Gamma_n$. The terminal groups are the groups labeling the vertices at level $0$.
\end{df}

It is shown in \cite[Theorem 1.4, Lemma 7.3]{WZ17} and \cite{E} that if $G$ is a hyperbolic or toral relatively hyperbolic virtually compact special group then $G$ has a finite index subgroup $G_0$ with a malnormal hierarchy and it passes to the $\widehat{G_0}$ in such a way that $\widehat{G_0}$ acts $1$-acylindrically on a profinite tree $T$ (see \cite[Theorem 3.3, Theorem 4.2]{WZ17}).

Next we shall need the profinite version of VRC property for the profinite completion of a virtually special groups, so we dedicate a subsection to this.

\subsection{Profinite VRC groups}
A group $G$ is called {\it VRC} if every cyclic subgroup of $G$ is a virtual retract of $G$. It is a very important property in geometric group theory that was introduced by Long and  
Reid explicitly in \cite{LR}, however, implicitly they were investigated much earlier. Minasyan made systematic study of this property in \cite{M} where he emphasized that property
VRC in general, apart from having numerous applications, is also very interesting by themselves.

We introduce here the profinite version of this property. It does not transport to the profinite case well literally since a torsion-free infinite procyclic group might be an infinite direct product of procyclic pro-$p$ groups, so we shall use procyclic pro-$p$ groups for each $p$ instead.

\begin{df}
We say that a profinite group $G$ is {\it VRC} if every procyclic pro-$p$ group of $G$ is a virtual retract of $G$.
\end{df} 

Note that for profinite groups VRC passes to subgroups, since for any subgroup $H$ of a VRC profinite group $G$, $U \cap H$ is open in $H$ for any open subgroup $U$ of $G$.

The next lemma shows that absolutely torsion-free profinite groups (i.e., profinite groups whose all open subgroups have torsion-free abelianization) are VRC. Note that absolutely torsion-free property was introduced and studied by W\"urfel in \cite{Wu85} with motivation from Galois theory. Free pro-$p$ groups, free abelian pro-$p$ groups and pro-$p$ completions of surface groups are examples of absolutely torsion free pro-$p$ groups. W\"urfel proved that the class of absolutely torsion free pro-$p$ groups contains a $p$-Sylow subgroup of the absolute Galois group of a field containing all $p$-power roots of unity and is closed under forming inverse limits and free pro-$p$ products (see  \cite[Proposition 3]{Wu85}); moreover, he proved that a direct product of a free abelian pro-$p$ group and an absolutely torsion-free pro-$p$ group is absolutely torsion-free. In \cite{SZ} were characterized the right-angled Artin groups whose pro-$p$ completion is absolutely torsion free. 

\begin{lem}\label{lem:absolutely torsion free}
Let $G$ be an absolutely torsion-free profinite group. Then $G$ is VRC.
\end{lem}

\begin{proof}
Let $P \simeq \Z_p$. Since $P$ is the intersection of open subgroup $U$ of $G$ containing it, the image $P_U$ of $P$ in the abelianization $U^{ab}$ of $U$ is non-trivial, for some $U$, and since $U^{ab}$ is torsion-free this image is isomorphic to $\Z_p$. Since $U^{ab}$ is VRC, we may assume, replacing $U$ by its open subgroup, that we have a retraction $r_U: U^{ab} \to P_U$ whose restriction to $P$ is injective. Defining $r: U \to P$ as $r = (\varphi|_P)^{-1}r_U\varphi$, where $\varphi: U \to U^{ab}$ is the natural epimorphism, we get the needed retraction.
\end{proof}

The fact that right-angled Artin/Coxeter groups virtually are (VRC) has already been observed by Aschenbrenner, Friedl and Wilton (see \cite{AFW15}), based on an earlier result of Agol (see \cite{A}). The virtually was droped in \cite[Corollary 1.6]{M}. The next proposition is the profinite version of it: 

\begin{prop}\label{prop:virtualretractions}
The profinite completion of a right-angled Artin group is VRC. 
\end{prop}

\begin{proof}
Note that a right-angled Artin group $G$ have a hierarchy where vertex groups are retracts (see \cite[Lemma 3.4]{CRPZ25}). 

Consider the action of $\widehat{G}$ on its Bass--Serre tree $T$. First, if $P \simeq \Z_p$ does not fix a vertex, then $P$ acts freely on $T$. Since $P$ is the intersection of open subgroup $U$ of $\widehat{G}$ containing it, by \cite[Lemma 4.6]{LZ24} $\bigcap_U \widetilde U = 1$ and so we can choose an open $U$ containing $P$ such that $U \to U/\widetilde{U}$ is injective on $P$. Note that $U/\widetilde U = \pi_1(U \backslash T)$. Thus the image $\bar{P}$ of $P$ in $U/\widetilde{U}$ is isomorphic to $\Z_p$. Now by Lemma \ref{lem:absolutely torsion free}, $\bar{P}$ is a virtual retract of $U/\widetilde{U}$, so after replacing $U$ by a possibly smaller open subgroup, we obtain an epimorphism $U/\widetilde{U} \to \Z_p$ that restricts injectively to $\bar{P}$. Taking the composition, we obtain an epimorphism $\varphi: U \to \Z_p$ whose restriction to $P$ is injective. Set $C = \varphi(P)$ and $U_0 = \varphi^{-1}(C)$ which is open in $U$ and contains $P$. Defining $r: U_0 \to P$ as $r = (\varphi|_P)^{-1}\varphi|_{U_0}$, we get a retraction.

Let $\Z_p$ be contained in $\widehat{G_v}$ for some $v$. We will prove the statement by induction on the length of the  hierarchy of $\widehat{G}$. The terminal vertex groups are cyclic groups and since $\widehat{G_v}$ is a virtual retract. Composing both retractions, we finish the base case. In the general, if $\Z_p$ does not fix a vertex, by the previous result it is a virtual retract of $\widehat{G_v}$ and $\widehat{G_v}$ is a virtual retract of $\widehat{G}$; if $\Z_p$ fixes a vertex, then by induction on $\widehat{G_v}$, $\Z_p$ is a virtual retract of $\widehat{G_v}$ and $\widehat{G_v}$ is a virtual retract of $\widehat{G}$. This finishes the proof.
\end{proof}

\begin{coro}\label{cor:virtually special}
The profinite completion of a virtually compact special group is virtually VRC.
\end{coro}

\begin{proof}
A virtually compact special group contains a finite index subgroup $G_0$ that embeds into right-angled Artin group as virtual retract and so $\widehat{G_0}$ embeds into the profinite completion of a right-angled Artin group (see \cite[Proposition 3.2, Proposition 3.3]{HW10}). Since the property we want to prove passes to subgroups, we deduce the result form Proposition \ref{prop:virtualretractions}.
\end{proof}

Combining Theorem \ref{theo:pro-pi-stabilizers-profinite-version} and Corollary \ref{cor:virtually special} we prove Theorem \ref{theo:virtually compact special case}.
\\

Consider the following list of groups:
\begin{enumerate}
\item[(i)] limit groups. In particular $\pi_1(S_g)$, where $S_g$ is a closed orientable surface of genus $g \geq 2$,

\item[(ii)] fundamental group of hyperbolic $3$-manifolds,

\item[(iii)] finitely generated Fuchsian non-triangle group,

\item[(iv)] word-hyperbolic graph braid groups,

\item[(v)] one-relator groups with torsion,

\item[(vi)] finitely generated Kleinian groups.

\item[(vii)] Hyperbolic hyperbolic-by-cyclic groups, i.e., a semidirect product $G \rtimes\Z$ of a hyperbolic group $G$ with infinite cyclic group,

\item[(viii)] the fundamental group of standard arithmetic hyperbolic manifolds.
\end{enumerate}

\begin{itemize}
\item Note that all these groups are virtually compact special (see \cite{Dan21}, \cite{AFW15}, \cite{CW04, Gen21}, \cite{EEK82}, \cite{MZ13}, \cite{DMM25}, \cite{BFW11, PB24}). The groups in (iv) are, in fact, compact special.

\item Note that (i), (ii) and (iii) have acylindrical actions on their respective Bass--Serre trees (see \cite{ZZ19}, \cite{WZ17}, \cite{LZ24}).

\item Note that the groups (ii), (iii), (v) are word-hyperbolic (see \cite{BCR16}, \cite{MZ13}). The groups in (i) and (vi) are relatively hyperbolic (see \cite{AB06}, \cite{Bow12}). By \v{S}varc-Milnor lemma, the groups in (viii) are word-hyperbolic if the manifold is compact since it acts (by isometries) properly and cocompactly in $\mathbb{H}^n$ and in the non-compact case it is relatively hyperbolic (see \cite[Example 1]{Szc98}).
\end{itemize}

Consider the groups in (i), (ii) and (iii). Since they are virtually compact special, by Corollary \ref{cor:virtually special}, the profinite completions of all these groups are VRC. Then Corollary \ref{coro:application} follows from Theorem \ref{theo:pro-pi-stabilizers-profinite-version}.

Now, the groups (iv), (v), (vi), (vii) and (viii) have a finite-index subgroup $G_0$ having a malnormal quasiconvex hierarchy and $\widehat{G_0}$ acts on its Bass--Serre tree. Then Corollary \ref{coro:application2} follows from Theorem \ref{theo:virtually compact special case}. 

We finish the paper with the following:

\begin{conj}
Let $G$ be a torsion free hyperbolic virtually compact special group and $H$ a finitely generated pro-$\calC$ subgroup of $\widehat{G}$, where $\calC$ does not contain all finite groups. Then $H$ is projective.
\end{conj}

This conjecture is true if $H$ is prosoluble (see \cite[Theorem B]{LZ24}).

\end{document}